\documentclass[]{theclass}
\begin{document}

\begin{frontmatter}

\titledata{Betwixt and between 2-factor Hamiltonian and Perfect-Matching-Hamiltonian graphs}{}           

\authordata{Federico Romaniello}
{Dipartimento di Matematica, Informatica ed Economia\\ Universit\`{a} degli Studi della Basilicata, Italy}{federico.romaniello@unibas.it}{}
{}

\authordata{Jean Paul Zerafa}
{St. Edward's College, Triq San Dwardu\\ Birgu (Citt\`{a} Vittoriosa), BRG 9039, Cottonera, Malta; \\ Department of Technology and Entrepreneurship Education\\University of Malta, Malta; \\
Department of Computer Science, Faculty of Mathematics, Physics and Informatics\\ Comenius University, Mlynsk\'{a} Dolina, 842 48 Bratislava, Slovakia}
{zerafa.jp@gmail.com}
{The author was partially supported by VEGA 1/0743/21, VEGA 1/0727/22, and APVV-19-0308.}


\begin{abstract}
A Hamiltonian graph is 2-factor Hamiltonian (2FH) if each of its 2-factors is a Hamiltonian cycle. A similar, but weaker, property is the Perfect-Matching-Hamiltonian property (PMH-property): a graph admitting a perfect matching is said to have this property if each one of its perfect matchings (1-factors) can be extended to a Hamiltonian cycle. It was shown that the \emph{star product} operation between two bipartite 2FH-graphs is necessary and sufficient for a bipartite graph admitting a 3-edge-cut to be 2FH. The same cannot be said when dealing with the PMH-property, and in this work we discuss how one can use star products to obtain graphs (which are not necessarily bipartite, regular and 2FH) admitting the PMH-property with the help of \emph{malleable} vertices, which we introduce here. We show that the presence of a malleable vertex in a graph implies that the graph has the PMH-property, but does not necessarily imply that it is 2FH. 
It was also conjectured that if a graph is a bipartite cubic 2FH-graph, then it can only be obtained from the complete bipartite graph $K_{3,3}$ and the Heawood graph by using star products. Here, we show that a cubic graph (not necessarily bipartite) is 2FH if and only if all of its vertices are malleable. We also prove that the above conjecture is equivalent to saying that, apart from the Heawood graph, every bipartite cyclically 4-edge-connected cubic graph with girth at least 6 having the PMH-property admits a perfect matching which can be extended to a Hamiltonian cycle in exactly one way. Finally, we also give two necessary and sufficient conditions for a graph admitting a 2-edge-cut to be: (i) 2FH, and (ii) PMH.
\end{abstract}

\end{frontmatter}

\section{Introduction}\label{section intro}
Graphs considered in the sequel are connected (unless otherwise stated) and are allowed to have multiedges but no loops. The vertex set and the edge set of a graph $G$ are denoted by $V(G)$ and $E(G)$, respectively, and the number of neighbours of a vertex $v$ is denoted by $\textrm{deg}(v)$. For some integer $t\geq 1$, a \emph{$t$-factor} of a graph $G$ is a $t$-regular spanning subgraph of $G$ (not necessarily connected). In particular, a \emph{perfect matching} of a graph is the edge set of a 1-factor, and a connected 2-factor of a graph is a \emph{Hamiltonian cycle}. For $k\geq 3$, a \emph{cycle} of length $k$ (or a $k$-cycle), denoted by $(v_{1}, \ldots, v_{k})$, is a sequence of mutually distinct vertices $v_1,v_2,\ldots,v_{k}$ with corresponding edge set $\{v_{1}v_{2}, \ldots, v_{k-1}v_{k}, v_{k}v_{1}\}$. For other definitions not explicitly stated here we refer the reader to \cite{BM}. A graph $G$ admitting a perfect matching is said to have the \emph{Perfect-Matching-Hamiltonian property} (for short, the PMH-property) if every perfect matching $M$ of $G$ can be \emph{extended} to a Hamiltonian cycle of $G$, that is, there exists a perfect matching $N$ of $G$ such that $M\cup N$ induces a Hamiltonian cycle of $G$. For simplicity, a graph admitting the PMH-property is said to be PMH or a PMH-graph. This property was introduced in the 1970s by Las Vergnas \cite{LasVergnas} and H\"{a}ggkvist \cite{Haggkvist} and for recent results in this area we suggest the following non-exhaustive list \cite{pmhlinegraphs,papillon,rook,ThomassenEtAl,Fink,cpcq,accordions,Kreweras}.  If we restrict ourselves to the class of $3$-regular graphs (cubic graphs), there is already a known and well-studied class which are naturally PMH (as we shall see in Theorem \ref{theorem 2fh equiv malleable}). This is the class of cubic $2$-factor Hamiltonian graphs. The term \emph{2-factor Hamiltonian} (2FH) was coined by Funk \emph{et al.} in \cite{2fh}, where the authors study Hamiltonian graphs with the property that all their $2$-factors are Hamiltonian. 
In their work, the authors prove that if a graph $G$ is a bipartite $t$-regular 2FH-graph, then $G$ is either a cycle or $t=3$. 

Before proceeding, we define what a star product is. Let $G_{1}$ and $G_{2}$ be two graphs each containing a vertex of degree $3$, say, $v_{1}\in V(G_{1})$ and $v_{2}\in V(G_{2})$. Let $x_{1},y_{1},z_{1}$ be the neighbours of $v_{1}$ in $G_{1}$, and $x_{2},y_{2},z_{2}$ be the neighbours of $v_{2}$ in $G_{2}$. A \emph{star product} on $v_{1}$ and $v_{2}$, denoted by $G_{1}(x_{1}y_{1}z_{1})*G_{2}(x_{2}y_{2}z_{2})$, is a graph operation that consists in constructing the new graph $(G_{1}-v_{1})\cup (G_{2}-v_{2})\cup\{x_{1}x_{2}, y_{1}y_{2}, z_{1}z_{2}\}$. The $3$-edge-cut $\{x_{1}x_{2}, y_{1}y_{2}, z_{1}z_{2}\}$ is referred to as the \emph{principal $3$-edge-cut} of the resulting graph (see for instance \cite{FouquetVanherpe}). Different graphs can be obtained by a star product on $v_{1}$ and $v_{2}$, for example, $(G_{1}-v_{1})\cup (G_{2}-v_{2})\cup\{x_{1}z_{2}, y_{1}y_{2}, z_{1}x_{2}\}$, but, unless otherwise stated, if it is irrelevant how the adjacencies in the principal 3-edge cut look like, we use the notation $G_{1}(v_{1})*G_{2}(v_{2})$ and we say that it is a graph obtained by a star product on $v_{1}$ and $v_{2}$. For simplicity, we shall also say that the resulting graph has been obtained by applying a star product \emph{between} $G_{1}$ and $G_{2}$. Since a star product between a graph $G$ and the unique cubic graph on two vertices results in $G$ itself, in the sequel we shall tacitly assume that when considering a star product between two graphs, neither one of the two graphs is the cubic graph on two vertices.

\begin{proposition}\cite{2fh}\label{prop starproduct 2fh}
Let $G=G_{1}(v_{1})*G_{2}(v_{2})$ be a bipartite graph which is obtained by a star product on $v_{1}\in V(G_{1})$ and $v_{2}\in V(G_{2})$, both of degree $3$. Then, $G$ is 2FH if and only if $G_{1}$ and $G_{2}$ are both 2FH.
\end{proposition}

We note that in the above proposition, $G_{1}$ and $G_{2}$ are not necessarily cubic graphs, and only need to admit a vertex of degree $3$ each, denoted above by $v_{1}$ and $v_{2}$, respectively. Moreover, we remark that, in the above proposition, the hypothesis that $G$ is bipartite is needed, because although the complete graph $K_{4}$ is a 2FH-graph, the graph obtained by applying a star product between two copies of $K_{4}$ is not 2FH (and neither PMH).
By using Proposition \ref{prop starproduct 2fh}, the authors construct an infinite family of bipartite cubic 2FH-graphs by taking repeated star products of $K_{3,3}$ and the Heawood graph. For example, for each $i\in\{1,2,3\}$, let $G_i$ be a copy of $K_{3,3}$ or the Heawood graph, and let $v_i\in V(G_i)$. The graph $\left(G_1(v_1)*G_2(v_2)\right)*G_3(v_3)$ is a graph obtained by repeated star products of $K_{3,3}$ and the Heawood graph. In \cite{2fh}, the authors also conjecture that these are the only bipartite cubic 2FH-graph, and this conjecture is still widely open. 

\begin{conjecture}[Funk \emph{et al.}, 2003 \cite{2fh}]\label{conjecture 2fh}
A bipartite cubic 2FH-graph can be obtained from the complete bipartite graph $K_{3,3}$ and the Heawood graph by repeated star products.
\end{conjecture}	

\section{Malleable vertices}

Let $\partial v$ be the set of edges incident to a vertex $v$.

\begin{definition}
Let $G$ be a graph admitting a perfect matching and let $v$ be a vertex of $G$ having degree $t\geq 2$. The vertex $v$ is said to be \emph{$t$-malleable} (or just \emph{malleable}) if for every perfect matching $M$ of $G$, there exist Hamiltonian cycles $H_{1},\ldots, H_{t-1}$ all extending $M$, such that $\partial v-M\subset \cup_{i=1}^{t-1}E(H_{i})$.
\end{definition}

Therefore, if $G$ admits a $t$-malleable vertex $v$, given a perfect matching $M$ of $G$, there exist $t-1$ distinct Hamiltonian cycles, such that each Hamiltonian cycle extends $M$ and contains a different edge of $\partial v-M$, implying that the $t-1$ Hamiltonian cycles cover all edges incident to  $v$ (since every Hamiltonian cycle contains the edge in $\partial v\cap M$). Moreover, if a graph admits a malleable vertex, then it clearly is PMH. In particular, if $|V(G)|>2$ and $v\in V(G)$ is malleable, then the number of neighbours of $v$ must be equal to $\textrm{deg}(v)$, that is, there cannot be any multiedges incident to $v$. Although the definition of malleable vertices seems quite strong, in even cycles and cubic graphs, the presence of a malleable vertex is equivalent to saying that the graph is 2FH.

\subsection{Even cycles and cubic graphs}\label{section malleable cubic}
A (connected) 2-regular graph admitting a 2-malleable vertex, must be bipartite, otherwise it does not admit a perfect matching. One can easily see that cycles on an even number of vertices are 2FH and all the vertices are 2-malleable. So consider cubic graphs.
\begin{theorem}\label{theorem 2fh equiv malleable}
A cubic graph $G$ is 2FH if and only if $G$ admits a 3-malleable vertex.
\end{theorem} 

\begin{proof}
($\Rightarrow$) Let $u$ be a vertex of $G$ and let $M$ be a perfect matching of $G$. Moreover, let $\overline{M}=E(G)-M$, that is, the edge set of the complementary $2$-factor of $M$. Since $G$ is 2FH, $\overline{M}$ gives a Hamiltonian cycle, and since $G$ is of even order, $E(\overline{M})=N_{1}\cup N_{2}$, where $N_{1}$ and $N_{2}$ are edge-disjoint perfect matchings of $G$. Once again, since $G$ is 2FH, $M\cup N_{1}$ and $M\cup N_{2}$ are both Hamiltonian cycles of $G$. Thus, $u$ is a 3-malleable vertex.

($\Leftarrow$) Let $v$ be a 3-malleable vertex of $G$ and let $M_{1}$ be a perfect matching of $G$. We are required to show that $\overline{M_{1}}$ (the edge set of the complementary 2-factor of $M_{1}$) gives a Hamiltonian cycle. Since $G$ contains a 3-malleable vertex, it is PMH, and so there exists a perfect matching $M_{2}$ such that $M_{1}\cup M_{2}$ gives a Hamiltonian cycle of $G$. Let $M_{3}=E(G)-(M_{1}\cup M_{2})$ and let $\partial v=\{e_{1},e_{2},e_{3}\}$, such that $e_{i}\in M_{i}$, for each $i\in\{1,2,3\}$. Since $v$ is 3-malleable, there exists a Hamiltonian cycle of $G$ which extends $M_{3}$ and contains the edge $e_{2}$. Since $M_{1}\cup M_{2}$ forms a Hamiltonian cycle and $(M_{1}\cup M_{2})\cap M_{3}=\emptyset$, the only perfect matching of $G-M_{3}$ containing $e_{2}$ is $M_{2}$, and so $M_{2}\cup M_{3}$ (which is equal to $\overline
{M_{1}}$) forms a Hamiltonian cycle, as required.
\end{proof}

Since the vertex $u$ in the first part of the above proof was arbitrary, the next result clearly follows.

\begin{proposition}\label{prop all vertices are malleable}
Let $G$ be a cubic graph admitting a 3-malleable vertex. Then, $G$ is 2FH and all its vertices are 3-malleable.
\end{proposition}

Consequently, Theorem \ref{theorem 2fh equiv malleable} can be restated as follows: a cubic graph is 2FH if and only if all its vertices are malleable. In other words, either all or none of the vertices of a cubic graph are 3-malleable. Figure \ref{figureQ3notmalleable} depicts a perfect matching of the cube $\mathcal{Q}_{3}$ which can only be extended to a Hamiltonian cycle in exactly one way, and so, there is no vertex in $\mathcal{Q}_{3}$ which is 3-malleable. In fact, the cube is not 2FH (although it is PMH).

\begin{figure}[h]
\centering
\includegraphics[scale=1]{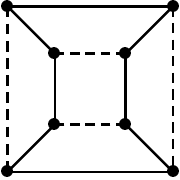}
\caption{$\mathcal{Q}_{3}$ does not admit any 3-malleable vertex since the dashed edges can be extended to a Hamiltonian cycle in exactly one way.}
\label{figureQ3notmalleable}
\end{figure}

In general, if a cubic PMH-graph $G$ (not necessarily bipartite) admits a perfect matching $M$ which extends to a Hamiltonian cycle in exactly one way (that is, there exists a unique perfect matching $N$ for which $M\cup N$ gives a Hamiltonian cycle), then the vertices of $G$ are not malleable, and so the graph is not 2FH (by Theorem \ref{theorem 2fh equiv malleable}). The converse of this statement is also true.

\begin{lemma}\label{lemma exactly one way}
Let $G$ be a cubic PMH-graph (not necessarily bipartite). The graph $G$ is not 2FH if and only if it admits a perfect matching which can be extended to a Hamiltonian cycle in exactly one way.
\end{lemma}

\begin{proof}
By the comment prior to the statement of the lemma, it suffices to prove the forward direction. Since $G$ is not 2FH, by Theorem \ref{theorem 2fh equiv malleable}, no vertex in $G$ is malleable. Let $v\in V(G)$ and let $\partial v=\{e_{1},e_{2},e_{3}\}$. Since $v$ is not malleable, there exists a perfect matching of $G$, say $M_{1}$, such that all perfect matchings $M_{2}$ of $G$ for which $M_{1}\cup M_{2}$ is a Hamiltonian cycle, intersect $\partial v - M_{1}$ in the same edge. Let $M_{2}$ be such a perfect matching and, without loss of generality, assume that $e_{1}$ and $e_{2}$ belong to $M_{1}$ and $M_{2}$, respectively. Since $M_{1}\cup M_{2}$ is a Hamiltonian cycle, $E(G)-(M_{1}\cup M_{2})$ is a perfect matching, say $M_{3}$, containing the edge $e_{3}$. Since $G$ is PMH, there exists a perfect matching $N$ of $G-M_{3}$, such that $N\cup M_{3}$ is a Hamiltonian cycle of $G$. Since $G-M_{3}$ is a connected even cycle, the perfect matching $N$ is either equal to $M_{1}$ or to $M_{2}$. By our assumption, $N$ cannot be equal to $M_{1}$, because otherwise there exists a Hamiltonian cycle extending $M_{1}$ which contains the edge $e_{3}$. Therefore, $N$ must be equal to $M_{2}$. Consequently, $M_{3}$ is a perfect matching of $G$ which can be extended to a Hamiltonian cycle of $G$ in exactly one way.
\end{proof}

Before proceeding, the following notions dealing with the cyclic connectivity of a graph require defining. An edge-cut $X$ is said to be \emph{cycle-separating} if at least two components of $G-X$ contain cycles. A (connected) graph $G$ is said to be cyclically $k$-edge-connected if $G$ admits no set with less than $k$ edges which is cycle-separating. Consider once again Conjecture \ref{conjecture 2fh}. As stated in \cite{2fh}, a smallest counterexample to this conjecture must be cyclically 4-edge-connected (see \cite{m1fcub}), and such a counterexample must have girth at least 6 (see \cite{charm1freg}). The authors of \cite{2fh} state that to prove this conjecture it suffices to show that the Heawood graph is the only bipartite cyclically 4-edge-connected cubic 2FH-graph of girth at least 6. However, thinking about cubic 2FH-graphs through malleable vertices and Lemma \ref{lemma exactly one way} suggests another way how one can look at Conjecture \ref{conjecture 2fh}. In fact, a smallest counterexample to this conjecture can be searched for in the class of bipartite cubic PMH-graphs (recall that Conjecture \ref{conjecture 2fh} deals with bipartite cubic graphs). By Lemma \ref{lemma exactly one way}, Conjecture \ref{conjecture 2fh} of Funk \emph{et al.} can be restated equivalently in terms of a strictly weaker property than 2-factor Hamiltonicity: the PMH-property.

\begin{conjecture}\label{conjecture new}
Every bipartite cyclically 4-edge-connected cubic PMH-graph with girth at least 6, except the Heawood graph, admits a perfect matching which can be extended to a Hamiltonian cycle in exactly one way.
\end{conjecture}

\subsection{Non-cubic graphs admitting a malleable vertex}

Even though Section \ref{section malleable cubic} may suggest otherwise, the existence of a malleable vertex in a graph does not necessarily imply that the graph is 2FH. In fact, we note that for every $t>3$, there exists a bipartite $t$-regular graph whose vertices are all $t$-malleable, but the graph itself is not 2FH (recall that in \cite{2fh} it was shown that there are no bipartite $t$-regular 2FH-graphs for $t>3$). Consider, for example, the complete bipartite graphs $K_{t,t}$ for every $t>3$. Also, for every odd $t>3$, the vertices of the complete graph $K_{t+1}$ are all $t$-malleable, but the graph is not 2FH.

Graphs admitting a malleable vertex which are not 2FH are not necessarily regular. In fact, consider the graph $\mathcal{Y}_{2n+1}$ obtained by adding a new vertex $v_{0}$ to the complete graph $K_{2n+1}$, for some $n\geq 2$, such that $v_{0}$ is adjacent to exactly three vertices of $K_{2n+1}$ (see Figure \ref{figure k5+v}).
\begin{proposition}
The graph $\mathcal{Y}_{2n+1}$ is PMH but not 2FH. Moreover, the vertex $v_0$ is 3-malleable.
\end{proposition}
\begin{proof}
Let $V(K_{2n+1})=\{v_{1},\ldots, v_{2n+1}\}$ and, without loss of generality, let the neighbours of $v_{0}$ in $\mathcal{Y}_{2n+1}$ be $v_{1},v_{2},v_{3}$. Then, the two disjoint cycles $(v_{0}, v_{1}, v_{2})$ and $(v_{3},v_{4},\ldots,\linebreak v_{2n+1})$ form a 2-factor, making the graph not 2FH. We also claim that the vertex $v_{0}$ is a 3-malleable vertex. In fact, let $M$ be a perfect matching of $\mathcal{Y}_{2n+1}$ and, without loss of generality, assume that $v_{0}v_{1}\in M$. If one can show that $\mathcal{Y}_{2n+1}-v_{0}v_{2}$ and $\mathcal{Y}_{2n+1}-v_{0}v_{3}$ each admit a Hamiltonian cycle extending $M$, then this would imply that $v_{0}$ is a 3-malleable vertex. Without loss of generality, consider $\mathcal{Y}_{2n+1}-v_{0}v_{2}$. Since $\mathcal{Y}_{2n+1}-v_{0}v_{2}$ contains a copy of the complete graph $K_{2n+1}$, there exists a Hamiltonian path of $\mathcal{Y}_{2n+1}-\{v_{0}\}$ with endvertices $v_1$ and $v_3$ which contains all the edges of $M-\{v_0v_1\}$. This latter path together with the edges $v_0v_1$ and $v_0v_3$ gives a Hamiltonian cycle of $\mathcal{Y}_{2n+1}-v_{0}v_{2}$ extending $M$. By a similar reasoning, $\mathcal{Y}_{2n+1}-v_{0}v_{3}$ admits a Hamiltonian cycle extending $M$. Since $M$ was arbitrary, the vertex $v_{0}$ is 3-malleable.
\end{proof}



\begin{figure}[h]
\centering
\includegraphics[scale=1]{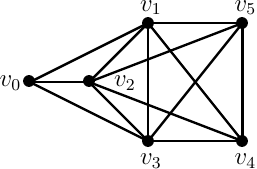}
\caption{The graph $\mathcal{Y}_{5}$.}
\label{figure k5+v}
\end{figure}


The above construction also provides us with graphs which are not 2FH, admit a malleable vertex, but not all of its vertices are as such (unlike even cycles and cubic graphs). In fact, let $M$ be a perfect matching of $\mathcal{Y}_{2n+1}$ containing the edges $v_{0}v_{1}$ and $v_{2}v_{3}$. Any Hamiltonian cycle of $\mathcal{Y}_{2n+1}$ extending $M$ cannot contain $v_{1}v_{2}$ or $v_{1}v_{3}$, and so, in particular, the vertex $v_{1}$ is not $(2n+1)$-malleable.

\begin{figure}[h]
\centering
\includegraphics[scale=1]{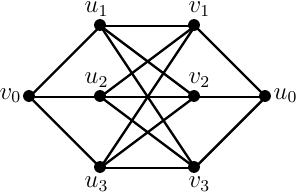}
\caption{The graph $\mathcal{B}_{3}$.}
\label{figure k33+2}
\end{figure}

We can also obtain examples of graphs which are not 2FH and admit a malleable vertex which are bipartite and less dense (with respect to the number of edges) in the following way. For every $n\geq 3$, let $\mathcal{B}_{n}$ be the bipartite graph with partite sets $\{u_{0}, u_{1}, \ldots, u_{n}\}$ and $\{v_{0}, v_{1}, \ldots, v_{n}\}$, such that $E(\mathcal{B}_{n})=\{u_{0}v_{1}, u_{0}v_{2}, u_{0}v_{3}\}\cup\{v_{0}u_{1}, v_{0}u_{2}, v_{0}u_{3}\}\cup\{u_{i}v_{j}: \textrm{ for any }i,j\in[n]\}$, where $[n]=\{1,\ldots, n\}$.

\begin{proposition}
For every $n\geq 3$, the graph $\mathcal{B}_{n}$ is PMH but not 2FH. Moreover, the vertices $u_{0}$ and $v_{0}$ are 3-malleable.    
\end{proposition}

\begin{proof}
Consider the  cycles: 
\begin{enumerate}[(i)]
\item $(u_{0},v_{1},u_{2},v_{2})$ and $(v_{0},u_{1},v_{3},u_{3})$, when $n=3$; and
\item $(u_{0},v_{1},u_{2},v_{2})$ and $(v_{0}, u_{1}, v_{3}, u_{4}, v_{4}, \ldots, v_{n}, u_{3})$, when $n>3$.
\end{enumerate}
In each case a disconnected 2-factor of $\mathcal{B}_{n}$ is formed, where, in particular, $v_{4}$ is followed by $u_{5}$ when $n>4$. Consequently, $\mathcal{B}_{n}$ is not 2FH. Next, we show that $u_{0}$ and $v_{0}$ are 3-malleable, which implies that $\mathcal{B}_{n}$ is PMH, for every $n\geq 3$. Let $M$ be a perfect matching of $\mathcal{B}_{n}$, and without loss of generality, assume that $\{u_{0}v_{1}, v_{0}u_{1}\}\subset M$. Due to the symmetry of $\mathcal{B}_{n}$, without loss of generality, we can further assume that exactly one of the following occurs:
\begin{enumerate}[(i)]
\item $\{u_{2}v_{2}, u_{3}v_{3}\}\subset M$;
\item $u_{2}v_{2}\in M$ and $u_{3}v_{3}\not\in M$; and
\item $u_{2}v_{2}, u_{2}v_{3}, u_{3}v_{3}, u_{3}v_{2}$ do not belong to $M$.
\end{enumerate}
We note that the last two instances only occur when $n>3$. Let $M'=M-\{u_{0}v_{1}, v_{0}u_{1}\}$. The graph $\mathcal{B}'=\mathcal{B}_{n}-\{u_{0},u_{1},v_{0},v_{1}\}$ is isomorphic to the complete bipartite graph $K_{n-1,n-1}$ and $M'$ is one of its perfect matchings. Since every vertex in $\mathcal{B}'$ is $(n-1)$-malleable, there exist Hamiltonian cycles $H_{1}'$ and $H_{2}'$ of $\mathcal{B}'$, both extending $M'$ in such a way that $u_{2}v_{3}\in E(H_{1}')$ and $u_{3}v_{2}\in E(H_{2}')$. 

In the first case, the following set of edges gives a Hamiltonian cycle of $\mathcal{B}_{n}$ which extends $M$ and contains $u_{2}v_{0}$ and $u_{0}v_{3}$: $\left (E(H_{1}')-\{u_{2}v_{3}\}\right )\cup \{u_{2}v_{0}, v_{0}u_{1}, u_{1}v_{1}, v_{1}u_{0}, u_{0}v_{3}\}$. In the second case, $\left (E(H_{2}')-\{u_{3}v_{2}\}\right )\cup \{u_{3}v_{0}, v_{0}u_{1}, u_{1}v_{1}, v_{1}u_{0}, u_{0}v_{2}\}$ is the edge set of a Hamiltonian cycle of $\mathcal{B}_{n}$ which extends $M$ and contains $u_{3}v_{0}$  and $u_{0}v_{2}$. Hence, $u_{0}$ and $v_{0}$ are both 3-malleable. 
\end{proof}

We note that the 3-malleability of $u_{0}$ and $v_{0}$ cannot be proved by using Las Vergnas' Theorem \cite{LasVergnas}, and that Theorem 2 in \cite{Yang} can only be used for the cases when $n=3$ or $4$. Furthermore, the graph $\mathcal{B}_{n}$ can be turned into a non-bipartite non-regular graph which is not 2FH and admits a 3-malleable vertex by adding the edge $v_{n-1}v_{n}$.

Finally, we also remark that if a graph is 2FH, it does not mean that all its vertices are malleable (as in the case of even cycles and cubic graphs). An example of such a graph is $K_{3,3}$ with an edge $e$ added between two vertices of the same partite set---we denote this graph by $K_{3,3}+e$. Since there is no perfect matching of $K_{3,3}+e$ which contains $e$, the graph $K_{3,3}+e$ is 2FH (since $K_{3,3}$ is 2FH). However, the endvertices of the edge $e$ are not 4-malleable.
 
The reason why the construction of the graphs $\mathcal{B}_{n}$ and $\mathcal{Y}_{2n+1}$ was given is because the two classes of graphs contain 3-malleable vertices and so can be used in the general results proven in Section \ref{section pmh from smaller} to obtain PMH-graphs with arbitrarily large maximum degree by using star products.

\section{Star products and PMH-graphs}
In this section, we study what happens when we look at star products between PMH-graphs which are not necessarily bipartite and 2FH as in \cite{2fh}. We find general ways how one can obtain PMH-graphs (not necessarily cubic) from smaller graphs by using star products. This is done by the help of malleable vertices. Although there is a clear connection between 2FH-graphs and PMH-graphs, an analogous result to Proposition \ref{prop starproduct 2fh} for PMH-graphs is not possible, as the following section on cubic graphs shows.

\subsection{Cubic graphs revisited}

\begin{proposition}\label{prop pmh star product}
Let $G_{1}$ and $G_{2}$ be two cubic graphs, and let $u\in V(G_{1})$ and $v\in V(G_{2})$.
\begin{enumerate}[(i)]
\item If $G_{1}(u)*G_{2}(v)$ is PMH, then $G_{1}$ and $G_{2}$ are PMH.
\item The converse of (i) is not true.
\end{enumerate}
\end{proposition}

\begin{proof}
(i) First assume that $G_{1}(u)*G_{2}(v)$ is PMH and let $X=\{u_{1}v_{1}, u_{2}v_{2},u_{3}v_{3}\}$ be the principal $3$-edge-cut of $G_{1}(u)*G_{2}(v)$, where $u_{1},u_{2},u_{3}$ are the neighbours of $u$ in $G_{1}$, and $v_{1},v_{2},v_{3}$ are the neighbours of $v$ in $G_{2}$. Let $M$ be a perfect matching of $G_{1}$, and without loss of generality, assume that $u_{1}u\in M$. Let $M'$ be a perfect matching of $G_{1}(u)*G_{2}(v)$ containing $u_{1}v_{1}$ and $M-\{u_{1}u\}$. We remark that such a perfect matching exists, since, in particular, every edge of a bridgeless cubic graph is contained in a perfect matching (see \cite{Sch}). Furthermore, since $G_{1}(u)*G_{2}(v)$ is PMH, $M'$ (and every other perfect matching of this graph) intersects $X$ in exactly one edge, and there exists a Hamiltonian cycle $H$ of $G_{1}(u)*G_{2}(v)$ extending $M'$ and containing exactly one of the edges $u_{2}v_{2}$ and $u_{3}v_{3}$. Assume $u_{2}v_{2}\in E(H)$. This means that $H$ induces a path in $G_{1}$ having end-vertices $u_{1}$ and $u_{2}$, passes through all the vertices in $V(G_{1})-\{u\}$ and contains $M-\{u_{1}u\}$. This path together with the edges $u_{1}u$ and $u_{2}u$ forms a Hamiltonian cycle of $G_{1}$ extending $M$. Hence, $G_{1}$ is PMH, and by a similar reasoning, one can show that $G_{2}$ is also PMH.

(ii) Let $G_{1}$ and $G_{2}$ be two copies of the cube, and let $u\in V(G_{1})$ and $v\in V(G_{2})$. Both $G_{1}$ and $G_{2}$ are PMH (by \cite{Fink}), but $G_{1}(u)*G_{2}(v)$ is not. In fact, consider the perfect matching of $G_{1}(u)*G_{2}(v)$ shown in Figure \ref{figureQ3Q3}. One can clearly see that it cannot be extended to a Hamiltonian cycle.
\end{proof}

\begin{figure}[h]
\centering
\includegraphics[scale=1]{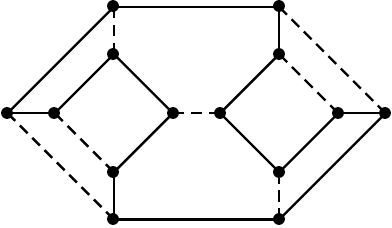}
\caption{A star product between two copies of the cube. The dashed edges cannot be extended to a Hamiltonian cycle.}
\label{figureQ3Q3}
\end{figure}

The second part of the above proof shows that unlike Proposition \ref{prop starproduct 2fh}, a star product between two bipartite PMH-graphs does not guarantee that the resulting graph is PMH. 

\begin{corollary}\label{cor 3cutconn}
If $G$ is a cubic PMH-graph having a $3$-edge-cut, then $G$ can be obtained by an appropriate star product between two cubic PMH-graphs $G_{1}$ and $G_{2}$.
\end{corollary}

The above corollary (and also Conjecture \ref{conjecture new}) are the main reasons why in \cite{papillon}, the study of cubic graphs which are PMH or just even-2-factorable was restricted to graphs having girth at least $4$. In \cite{papillon}, a graph $G$ is defined to be \emph{even-2-factorable} (for short E2F) if each of its 2-factors consist only of even cycles. When $G$ is cubic, $G$ is E2F if and only if each of its perfect matchings can be extended to a 3-edge-colouring (see Figure \ref{figure papillon}). We note that if a cubic graph is PMH then it is even-2-factorable as well, but the converse is not necessarily true. 

\begin{figure}[h]
      \centering
      \includegraphics[scale=0.6]{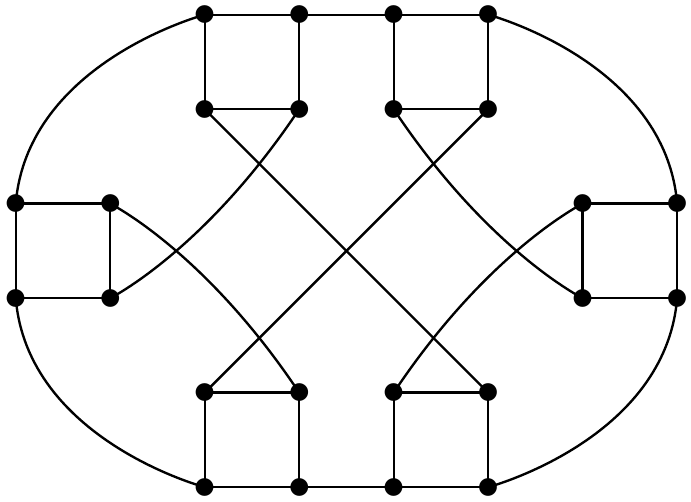}
	  \caption{An example of a papillon graph: an even-2-factorable cubic graph given in \cite{papillon}.}
      \label{figure papillon}
\end{figure}

As in Corollary \ref{cor 3cutconn}, a cubic graph having girth 3 which is also even-2-factorable (not necessarily PMH), can be obtained  by applying a star product between an even-2-factorable cubic graph and the complete graph $K_{4}$ (see \cite{papillon} for more details). Applying a star product between a graph and $K_{4}$ is also known as applying a $Y$-extension, which can be seen as expanding a vertex into a triangle (see Figure \ref{figure Yoperations}). We remark that the results given in the sequel do not necessarily yield PMH-graphs having girth 3, as Remark \ref{remark Girth4+PMH} shows.

\begin{figure}[h]
      \centering
      \includegraphics[scale=1]{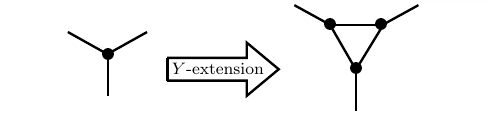}
	  \caption{$Y$-extension.}
      \label{figure Yoperations}
\end{figure}

Despite the discouraging statement of Proposition \ref{prop pmh star product}, one can still obtain PMH-graphs from smaller PMH-graphs by using a star product (or repeated star products) and 3-malleable vertices, as we shall see in the following section.

\subsection{Obtaining PMH-graphs from smaller graphs}\label{section pmh from smaller}

Before proceeding we give the following definition. Following the notation in \cite{lovasz}, an edge-cut in a graph $G$ admitting a perfect matching is said to be \emph{tight} if every perfect matching of $G$ intersects it in exactly one edge (not necessarily the same). 

\begin{lemma}\label{lem G1PMHG2malleable}
Let $G_{1}$ be a  PMH-graph admitting a vertex $u$ of degree 3 and let $G_{2}$ be a graph admitting a 3-malleable vertex $v$. The principal $3$-edge-cut of $G_{1}(u)*G_{2}(v)$ is tight if and only if $G_{1}(u)*G_{2}(v)$ is PMH.
\end{lemma}

\begin{proof}
Since $G_1$ and $G_2$ are both PMH-graphs, $G_1-u$ and $G_2-v$ are both of odd order. This implies that a perfect matching of $G_{1}(u)*G_{2}(v)$ cannot intersect its principal $3$-edge-cut in $2$ edges. Hence, if $G_{1}(u)*G_{2}(v)$ is PMH, its principal $3$-edge-cut is tight, as required. Consequently, it suffices to prove the forward direction. Let $X=\{u_{1}v_{1}, u_{2}v_{2},u_{3}v_{3}\}$ be the principal $3$-edge-cut of $G_{1}(u)*G_{2}(v)$, and assume that $X$ is tight. Let $M$ be a perfect matching of $G_{1}(u)*G_{2}(v)$, and let $u_{1},u_{2},u_{3}$ and $v_{1},v_{2},v_{3}$ be the neighbours of $u\in V(G_{1})$ and $v\in V(G_{2})$, respectively. Without loss of generality, assume that $M\cap X=\{u_{1}v_{1}\}$. Consequently, $M$ respectively induces perfect matchings $M_{1}$ and $M_{2}$ in $G_{1}$ and $G_{2}$, such that $u_{1}u\in M_{1}\subset E(G_{1})$, $v_{1}v\in M_{2}\subset E(G_{2})$, $M_{1}-\{u_{1}u\}\subset M$ and $M_{2}-\{v_{1}v\}\subset M$. Since $G_{1}$ is PMH, $M_{1}$ can be extended to a Hamiltonian cycle $H_{1}$ of $G_{1}$. Without loss of generality, we assume that $u_{2}u\in E(H_{1})$. Since $v$ is a 3-malleable vertex, $M_{2}$ can be extended to a Hamiltonian cycle $H_{2}$ of $G_{2}$ whose edge set intersects $v_{2}v$. Consequently, $(E(H_{1})-\{u_{1}u,u_{2}u\})\cup (E(H_{2})-\{v_{1}v,v_{2}v\})\cup \{u_{1}v_{1},u_{2}v_{2}\}$ is a Hamiltonian cycle of $G_{1}(u)*G_{2}(v)$ extending $M$, as required.
\end{proof}
This lemma shall be needed in the sequel when considering star products between PMH-graphs. The graph $G_{1}$ (similarly $G_{2}$) in Lemma \ref{lem G1PMHG2malleable} can either be bipartite or not, and in what follows we shall consider star products in two instances:
\begin{enumerate}[(i)]
\item between PMH-graphs with at least one being non-bipartite (Section \ref{section bipartite}); and
\item between non-bipartite PMH-graphs (Section \ref{section not bipartite}).
\end{enumerate}

We then finish this section with some examples of cubic PMH-graphs having small order (see Section \ref{section examples}).

\subsubsection{Star products between PMH-graphs with at least one being bipartite}\label{section bipartite}
Whilst a star product between two bipartite 2FH-graphs yields a 2FH-graph, Figure \ref{figureQ3Q3} shows that a star product between two bipartite PMH-graphs is not necessarily PMH. The example given in the figure is a star product between two copies of $\mathcal{Q}_{3}$, where the graph $\mathcal{Q}_{3}$ is itself PMH, but does not admit any 3-malleable vertex. The following proposition shows that the presence of a 3-malleable vertex in at least one of the two graphs between which a star product is applied guarantees the PMH-property in the resulting graph, given that at least one of the two initial graphs is bipartite.

\begin{proposition}\label{prop malleable+bip}
Let $G_{1}$ be a PMH-graph admitting a vertex $u$ of degree 3 and let $G_{2}$ be a graph admitting a 3-malleable vertex $v$. If at least one of $G_{1}$ and $G_{2}$ is bipartite, then $G_{1}(u)*G_{2}(v)$ is PMH.
\end{proposition}

\begin{proof}
Since at least one of $G_{1}$ and $G_{2}$ is bipartite, the principal $3$-edge-cut of $G_{1}(u)*G_{2}(v)$ is tight. The result follows by Lemma \ref{lem G1PMHG2malleable}.
\end{proof}

\begin{corollary}\label{cor two bipartite}
Let $G_{1}$ and $G_{2}$ be two bipartite graphs having the PMH-property such that $u$ is a vertex of degree 3 in $G_{1}$ and $v$ is a 3-malleable vertex in $G_{2}$. Then, $G_{1}(u)*G_{2}(v)$ is a bipartite PMH-graph.
\end{corollary}

We can extend the above corollary further. Let $G_{0}$ be a bipartite PMH-graph admitting $2$ vertices of degree 3, say $u_{1}$ and $u_{2}$. Furthermore, let $G_{1}$ and $G_{2}$ be bipartite PMH-graphs each admitting a 3-malleable vertex, say $v_{1}\in V(G_{1})$ and $v_{2}\in V(G_{2})$. By the previous corollary, $G_{0}(u_{1})*G_{1}(v_{1})$ is PMH. This graph is also bipartite, and so, reapplying a star product on the vertex corresponding to $u_{2}$ in $G_{0}(u_{1})*G_{1}(v_{1})$ and the vertex $v_{2}$ in $G_{2}$ gives a bipartite PMH-graph once again (by Corollary \ref{cor two bipartite}). For simplicity, we shall say that the resulting graph has been obtained by applying a star product on $u_{i}$ and $v_{i}$, for each $i\in\{1,2\}$. By repeating this argument we can state the following more general result.

\begin{theorem}\label{theorem bipartite}
Let $G_{0}$ be a bipartite PMH-graph admitting $t$ vertices of degree 3, for some $t\in\{1,\ldots, |V(G_{0})|\}$, say $u_{1},\ldots, u_{t}$. Let $\mathcal{I}\subseteq\{i:\textrm{deg}(u_{i})=3\}$. For each $i\in\mathcal{I}$, let $G_{i}$ be a bipartite graph admitting a 3-malleable vertex $v_{i}$. The bipartite graph obtained by applying a star product on $u_{i}$ and $v_{i}$, for each $i\in\mathcal{I}$, is PMH.
\end{theorem}

We remark that Theorem \ref{theorem bipartite} is best possible, in the sense that we cannot exchange the roles of the $u_{i}$s and the $v_{i}$s. In fact, if we assume that the $t$ vertices $u_{1},\ldots, u_{t}$ of $G_{0}$ are 3-malleable, and that, for each $i\in\mathcal{I}$, the graphs $G_{i}$ are PMH-graphs with the vertex $v_{i}$ being just a degree 3 vertex (and not 3-malleable), the same conclusion about the resulting graph cannot be obtained, as the following example in the class of cubic graphs shows.
Let $G_{0}$ be the graph $K_{3,3}$, and let $G_{1}$ and $G_{2}$ be two copies of the graph $\mathcal{Q}_{3}$. Let $u_1$ and $u_2$ be two vertices in $G_{0}$ belonging to the same partite set, and let $v_{1}\in V(G_{1})$ and $v_{2}\in V(G_{2})$. By Corollary \ref{cor two bipartite}, $G_{0}(u_1)*G_{1}(v_{1})$ is PMH. However, reapplying a star product on the vertex corresponding to $u_{2}$ in $G_{0}(u_{1})*G_{1}(v_{1})$ and the vertex $v_{2}$ of $G_{2}$ (that is, the final graph is obtained by applying a star product on $u_{i}$ and $v_{i}$, for each $i\in\{1,2\}$) does not yield a PMH-graph. Indeed, the dashed perfect matching portrayed in Figure \ref{figure k33 q3 q3} cannot be extended to a Hamiltonian cycle.

\begin{figure}[h]
\centering
\includegraphics[scale=1]{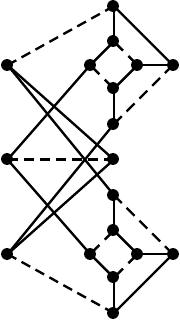}
\caption{The dashed edges cannot be extended to a Hamiltonian cycle. The above graph is $(K_{3,3} * \mathcal{Q}_3) * \mathcal{Q}_3$.}
\label{figure k33 q3 q3}
\end{figure}

 We thus move onto the next section and look at a star product between two non-bipartite PMH-graphs.

\subsubsection{Star products between non-bipartite PMH-graphs}\label{section not bipartite}

In Section \ref{section bipartite}, Proposition \ref{prop malleable+bip} already tells us that a star product between two PMH-graphs with exactly one being bipartite results in a PMH-graph (given that one of them admits a 3-malleable vertex). But what happens when both are non-bipartite? As already stated before, the graph obtained after applying a star product between two copies of the complete graph $K_{4}$ is not PMH, even though the graphs we started with, that is, the two copies of $K_{4}$, are both PMH. Given that $K_{4}$ is also 2FH, the previous example tells us that the presence of 3-malleable vertices alone does not guarantee the PMH-property in the resulting graph when both the PMH-graphs we start with are non-bipartite. In order to attain a general result about PMH-graphs obtained by applying a star product between two non-bipartite PMH-graphs, we extend Proposition \ref{prop malleable+bip} in a similar way as in Theorem \ref{theorem bipartite}.

\begin{theorem}\label{theorem 3ec}
Let $G_{0}$ be a bipartite PMH-graph of order $2n$ and with bipartition $U=\{u_{i}:i\in[n]\}$ and $V=\{v_{i}:i\in[n]\}$, for some $n>1$. Let $\mathcal{I}\subseteq\{i:\textrm{deg}(v_{i})=3\}$. For each $i\in\mathcal{I}$, let $G_{i}$ be a graph admitting a 3-malleable vertex $z_{i}$. The resulting graph $G$ obtained by applying a star product on $v_{i}$ and $z_{i}$, for each $i\in\mathcal{I}$, is PMH.
\end{theorem}

\begin{proof}
For each $i\in\mathcal{I}$, let $X_{i}$ be the (principal) $3$-edge-cut of $G$ arising from a star product on $v_{i}$ and $z_{i}$. This means that if $e\in X_{i}$, then one of the endvertices of $e$ belongs to $U$ and the other endvertex belongs to $V(G_{i}-z_{i})$. If $\mathcal{I}=\emptyset$, then $G$ is equal to $G_{0}$, and consequently, $G$ is PMH. So we can assume that $\mathcal{I}\neq\emptyset$. Let $M$ be a perfect matching of $G$, and let $E_{\mathcal{I}}$ be the collection of edges in $M$ having one endvertex in $U$ and one endvertex in some $V(G_{i}-z_{i})$, for $i\in\mathcal{I}$. 
Since $G_{0}$ is a bipartite PMH-graph, each $X_{i}$ is tight, and so, $|M\cap X_{i}|=1$ for each $i\in\mathcal{I}$. Consequently, $|E_{\mathcal{I}}|=|\mathcal{I}|$. Let $E_{\mathcal{I}}=\{e_{i}:i\in\mathcal{I}\}$, such that for each $i\in\mathcal{I}$, $e_{i}=x_{i}y_{i}$ for some $x_{i}\in U$ and $y_{i}\in V(G_{i}-z_{i})$. Moreover, for every $i\in\mathcal{I}$, let $f_{i}=x_{i}v_{i}$. The set of edges $M_{0}=\{f_{i}:i\in\mathcal{I}\}\cup M-\left (E_{\mathcal{I}}\bigcup\cup_{i\in\mathcal{I}}E(G_{i}-z_{i})\right )$ is a perfect matching of $G_{0}$, and since $G_{0}$ is PMH, there exists a Hamiltonian cycle $H_{0}$ of $G_{0}$ extending $M_{0}$. Without loss of generality, assume that $H_{0}$ is equal to $(u_{1},v_{1}, u_{2}, \ldots, v_{n})$, where $u_{2}$ is followed by $v_{2}$, and $v_{n}$ is preceded by $u_{n}$. Without loss of generality, assume further that $M_{0}=\{u_{i}v_{i}:i\in[n]\}$. In particular, this implies that for each $i\in\mathcal{I}$, $x_{i}=u_{i}$.

Before continuing, we remark that operations in the indices of the vertices $u_{i}$ are taken modulo $n$, with complete residue system $\{1,\ldots,n\}$. Let $j\in\mathcal{I}$, and let the neighbours of $u_{j}$ and $u_{j+1}$ belonging to $V(G_{j}-z_{j})$ be $\alpha_{j}$ and $\omega_{j}$. We note that $\alpha_{j}$ is equal to what we previously denoted by $y_{j}$. Given that $z_{j}$ is a 3-malleable vertex of $G_{j}$, there exists a Hamiltonian cycle $H_{j}$ of $G_{j}$ extending the perfect matching $\left (M\cap E(G_{j}-z_{j})\right )\cup \{z_{j}\alpha_{j}\}$, such that $z_{j}\alpha_{j}$ and $z_{j}\omega_{j}$ belong to $E(H_{j})$. Let $P_{j}$ be the path obtained after deleting the vertex $z_{j}$ from the cycle $H_{j}$. This process is repeated for every other integer in $\mathcal{I}$. We note that $\alpha_{j}$ and $\omega_{j}$ are the endvertices of $P_{j}$, and, in particular, by our assumption on $M_{0}$, we have $u_{j}\alpha_{j}\in M$, for every $j\in\mathcal{I}$.
By recalling that the edge set of the Hamiltonian cycle $H_{0}$ is $\cup_{i=1}^{n}\{u_{i}v_{i}, v_{i}u_{i+1}\}$ and the above considerations, one can deduce that the following edge set induces a Hamiltonian cycle of $G$ extending $M$: \[\cup_{i\in[n]-\mathcal{I}}\{u_{i}v_{i}, v_{i}u_{i+1}\}\bigcup\cup_{j\in\mathcal{I}}\left(\{u_{j}\alpha_{j}, \omega_{j}u_{j+1}\}\cup E(P_{j})\right ),\] as required.
\end{proof}

When $|\mathcal{I}|>1$, say $|I|=2$, and $G_{1}$ and $G_{2}$ are chosen to be non-bipartite, the above theorem shows that there do exist non-bipartite graphs such that when a star product is applied between them, the resulting graph is PMH. This follows because the graph obtained after appropriately applying a star product between $G_{0}$ and $G_{1}$ is non-bipartite.

We also remark that the reason why we cannot apply a star product on two adjacent vertices in $G_0$ is because the resulting graph is not necessarily PMH, as Figure \ref{figure k33 2 triangles} shows. In fact, applying a $Y$-extension to two adjacent vertices of $K_{3,3}$ results in a graph which does not have the PMH-property. Recall that $Y$-extensions can be explained in terms of a star product between a graph and $K_{4}$.

\begin{figure}[h]
\centering
\includegraphics[scale=1]{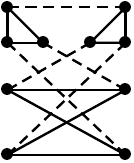}
\caption{The dashed edges cannot be extended to a Hamiltonian cycle.}
\label{figure k33 2 triangles}
\end{figure}

\subsubsection{Examples of cubic PMH-graphs having small order}\label{section examples}

In \cite{papillon}, it is stated that the papillon graph on 8 vertices is the smallest (with respect to the number of vertices) non-bipartite cubic graph with girth at least 4 which is even-2-factorable (E2F) and not PMH (recall that if a cubic graph is PMH then it is E2F). However, in what follows we present other non-bipartite cubic graphs on 4, 6, or 8 vertices which are PMH, and consequently E2F---these have girth strictly less than 4. Apart from $K_{4}$, there is another cubic graph on 4 vertices which is PMH, in particular, let $G$ be the unique bipartite cubic graph on 4 vertices (see Figure \ref{figure cubic bip on 4 vert}). By using the procedure outlined in Theorem \ref{theorem 3ec}, with $G_{0}=G$, $\mathcal{I}=\{1,2\}$ and $G_{1}=G_{2}=K_{4}$, we obtain a non-bipartite cubic PMH-graph on 8 vertices. This is equivalent to applying a $Y$-extension to two vertices belonging to the same partite set of $G$. Note that applying a $Y$-extension to a single vertex in $G$ also gives a PMH-graph, which is the unique non-bipartite cubic PMH-graph on 6 vertices (the second graph in Figure \ref{figure cubic bip on 4 vert}). 

\begin{figure}[h]
\centering
\includegraphics[scale=1]{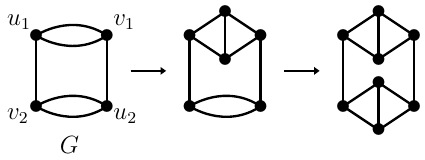}
\caption{Using Theorem \ref{theorem 3ec} to obtain non-bipartite cubic PMH-graphs.}
\label{figure cubic bip on 4 vert}
\end{figure}

Having said this, the graphs obtained by the methods given in the previous sections do not necessarily have girth 3, as the following remark shows.

\begin{remark}\label{remark Girth4+PMH}
An easy way to obtain cubic PMH-graphs with girth at least $4$ is the following. Let $F$ be the graph obtained by applying a $Y$-extension to a bipartite cubic 2FH-graph $F_{0}$ (see, for example, Figure \ref{figurek33k4}). Let $G$ be a bipartite cubic PMH-graph having no multiedges and let $v$ be a vertex of $F$ lying on its triangle. Since $F$ is a cubic 2FH-graph, $v$ is 3-malleable, and so, for any $u\in V(G)$, the graph $G(u)*F(v)$ is PMH by Proposition \ref{prop malleable+bip}. Moreover, $G(u)*F(v)$ has girth $4$. In fact, $G$ and $F_{0}$ do not have any multiedges and, since they are both bipartite, a cycle of length $3$ in $G(u)*F(v)$ can only occur if the edges of the cycle intersect (twice) the principal $3$-edge-cut of $G(u)*F(v)$, which is impossible. 
\begin{figure}[h]
\centering
\includegraphics[scale=1]{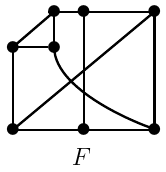}
\caption{Applying a $Y$-extension to $F_0=K_{3,3}$ from Remark \ref{remark Girth4+PMH}.}
\label{figurek33k4}
\end{figure}
Graphs obtained using this method are not necessarily 2FH. In fact, by letting $G=Q_{3}$ and $F_{0}=K_{3,3}$, the resulting graph depicted in Figure \ref{figure Girth4NOT2FH} is not 2FH, since the complementary 2-factor of the dashed perfect matching does not form a Hamiltonian cycle. 
\end{remark}

\begin{figure}[h]
\centering
\includegraphics[scale=1]{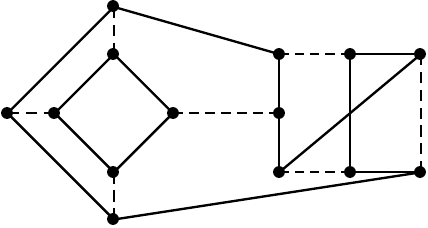}
\caption{A non-bipartite cubic PMH-graph having girth $4$ which is not 2FH.}
\label{figure Girth4NOT2FH}
\end{figure}

Although the above examples are cubic, we recall that the results in Section \ref{section pmh from smaller} can generate PMH-graphs (both bipartite and non-bipartite) with arbitrarily large maximum degree. This can be done by using graphs admitting a 3-malleable vertex as the ones portrayed after the proof of Theorem \ref{theorem 2fh equiv malleable} in Section \ref{section intro}. 

We also remark that despite the encouraging general methods obtained above, there are PMH-graphs admitting a 3-edge-cut, that is, obtained by using a star product (see Corollary \ref{cor 3cutconn}), which cannot be described by the methods portrayed so far. Such an example is given in Figure \ref{figure final example}. The graph denoted by $G_{1}*G_{2}$ (obtained by an appropriate star product on $u\in V(G_{1})$ and $v\in V(G_{2})$) is PMH. The graphs $G_{1}$ (bipartite) and $G_{2}$ (non-bipartite) are both PMH-graphs as well, however, $G_{1}$ and $G_{2}$ do not have any 3-malleable vertices and so, the reason why the resultant graph is PMH is not because of the above results, in particular, Theorem \ref{theorem 3ec} (see also Lemma \ref{lem G1PMHG2malleable}).
\begin{figure}[h]
\centering
\includegraphics[scale=1]{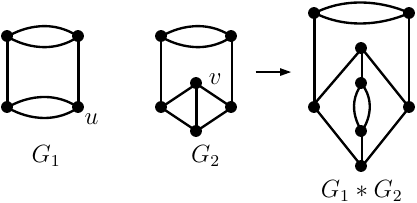}
\caption{The graphs $G_{1}$, $G_{2}$ and $G_{1}*G_{2}$ are all PMH.}
\label{figure final example}
\end{figure}
The three graphs given in Figure \ref{figure final example} contain 2-edge-cuts which we discuss in the next section with regards to PMH- and 2FH-graphs. In particular, although the PMH-property in the third graph given in Figure \ref{figure final example} cannot be explained by the previous theorems dealing with the star product, the reason behind it being PMH can be explained by Theorem \ref{theorem 2ecpmh} which gives a necessary and sufficient condition for a graph admitting a 2-edge-cut to be PMH. 

\section{2-edge-cuts in PMH- and 2FH-graphs}

Let $G_{1}$ and $G_{2}$ be two graphs (not necessarily regular), and let $e_{1}$ and $e_{2}$ be two edges such that $e_{1}=x_{1}y_{1} \in E(G_{1})$ and $e_{2}=x_{2}y_{2} \in E(G_{2})$. A \emph{$2$-cut connection} on $e_{1}$ and $e_{2}$ is a graph operation that consists of constructing the new graph $(G_{1}-e_{1})\cup (G_{2}-e_{2})\cup\{x_{1}x_{2}, y_{1}y_{2}\}$, and denoted by $G_{1}(x_{1}y_{1})\#G_{2}(x_{2}y_{2})$. The $2$-edge-cut $\{x_{1}x_{2}, y_{1}y_{2}\}$ is referred to as the \emph{principal $2$-edge-cut} of the resulting graph. It is clear that another possible graph obtained by a 2-cut connection on $e_{1}$ and $e_2$ is $G_{1}(x_{1}y_{1})\#G_{2}(y_{2}x_{2})$. Unless otherwise stated, if it is not important which of these two graphs is obtained, we use the notation $G_{1}(e_{1})\#G_{2}(e_{2})$ and we say that it is a graph obtained by a $2$-cut connection on $e_{1}$ and $e_{2}$. As in the case of star products, when this occurs, we say that the resulting graph has been obtained by applying a 2-cut connection \emph{between} $G_{1}$ and $G_{2}$. Given that graphs on an odd number of vertices do not admit a perfect matching, they cannot be studied with regards to the PMH-property. Since we shall be looking at the PMH-property of $G_{1}(e_{1})\#G_{2}(e_{2})$, $|V(G_{1})|$ and $|V(G_{2})|$ are either both odd or both even.  This is the reason why in the next theorem we shall assume that the graphs $G_{1}$ and $G_{2}$ are both of even order, as we are not only interested in whether $G_{1}(e_{1})\#G_{2}(e_{2})$ admits the PMH-property, but also whether $G_{1}$ and $G_{2}$ admit it. As we shall see, $G_1$ and $G_2$ will have a stronger property to guarantee the PMH-property in $G_{1}(e_{1})\#G_{2}(e_{2})$.

\begin{theorem}\label{theorem 2ecpmh}
Let $G=G_{1}(e_{1})\#G_{2}(e_{2})$ be a graph obtained by applying a 2-cut connection on $e_{1}\in E(G_{1})$ and $e_{2}\in E(G_{2})$, such that $G_{1}$ and $G_{2}$ both admit a perfect matching. Then, $G$ is PMH if and only if, for each $i\in\{1,2\}$, every perfect matching in $G_{i}$ can be extended to a Hamiltonian cycle of $G_{i}$ which contains $e_{i}$.
\end{theorem}

\begin{proof}
($\Rightarrow$) Let $e_{1}=x_{1}y_{1}$ and $e_{2}=x_{2}y_{2}$ be such that the principal 2-edge-cut of $G$ is $X=\{x_{1}x_{2}, y_{1}y_{2}\}$, and let $M_{1}$ be a perfect matching of $G_{1}$. Since $G$ is PMH, $G$ contains a perfect matching $M$, such that $M_{1}-\{e_{1}\}\subset M$, and, in particular, there exists a perfect matching $N$ of $G$ such that $M\cup N$ gives a Hamiltonian cycle of $G$. If $e_{1}\in M_{1}$, then $X\subset M$ and $N\cap X=\emptyset$. Consequently, the set of edges $N_{1}=N\cap E(G_{1})$ is a perfect matching of $G_{1}$, and $M_{1}\cup N_{1}$ gives a Hamiltonian cycle of $G_{1}$ containing $e_{1}$. Otherwise, if $e_{1}\not\in M_{1}$, then $M\cap X=\emptyset$ and $X\subset N$. Consequently, the set of edges $N_{1}=\left (N\cap E(G_{1})\right )\cup \{e_{1}\}$ is a perfect matching of $G_{1}$, and $M_{1}\cup N_{1}$ gives a Hamiltonian cycle of $G_{1}$ containing $e_{1}$, once again. By a similar argument one can show that every perfect matching of $G_{2}$ can be extended to a Hamiltonian cycle of $G_{2}$ containing $e_{2}$.

($\Leftarrow$) Conversely, assume that $M$ is a perfect matching of $G$. Notwithstanding whether $M$ contains the edges in $X=\{x_{1}x_{2}, y_{1}y_{2}\}$ or not, $M$ induces two perfect matchings $M_{1}\in E(G_{1})$ and $M_{2}\in E(G_{2})$ such that $M_{i}-\{e_{i}\}\subset M$, for each $i\in\{1,2\}$. Let $i\in\{1,2\}$. Note that $M\cap X=X$ if and only if $e_{i}\in M_{i}$. By our assumption, $G_{i}$ admits a Hamiltonian cycle $H_{i}$ which extends $M_{i}$ and contains $e_{i}$, and so, since $G_{i}$ is of even order, it admits a perfect matching $N_{i}$ such that $M_{i}\cup N_{i}=E(H_{i})$. Consequently, the edge set $M_{1}\cup N_{1}\cup M_{2}\cup N_{2}\cup X-\{e_{1},e_{2}\}$ gives a Hamiltonian cycle of $G$ containing $M$. Thus, $G$ is PMH as required.
\end{proof}

The above theorem explains the reason behind the PMH-property in all the three graphs shown in Figure \ref{figure final example}, not only the third. The graph $G_{1}$ is obtained by applying a 2-cut connection between two copies of the cubic graph on two vertices. The graph $G_{2}$ is obtained by applying a 2-cut connection between the cubic graph on two vertices and the graph $K_{4}$. The third graph denoted by $G_{1}*G_{2}$ in Figure \ref{figure final example} is obtained by applying an appropriate 2-cut connection between the cubic graph on two vertices and the graph $G_{2}$.

We also note that the condition in Theorem \ref{theorem 2ecpmh} that every perfect matching in $G_{i}$ has to be extended to a Hamiltonian cycle of $G_{i}$ containing the edge $e_{i}$ is required because, for example, a 2-cut connection between the cube $\mathcal{Q}_{3}$ (which is PMH) and any other appropriate PMH-graph does not yield a PMH-graph, since any perfect matching of the resulting graph containing the dashed edges cannot be extended to a Hamiltonian cycle, as can be seen in Figure \ref{figure 2cutq3}. 

\begin{figure}[h]
\centering
\includegraphics[scale=1]{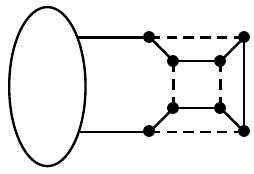}
\caption{A perfect matching containing the dashed edges cannot be extended to a Hamiltonian cycle.}
\label{figure 2cutq3}
\end{figure}

Next, we give a similar result to Theorem \ref{theorem 2ecpmh} but for 2FH-graphs. Before proceeding we note that, in general, if every 2-factor of a graph $G$ is a Hamiltonian cycle containing a particular edge $e\in E(G)$, then every 2-factor of $G$ containing $e$ is a Hamiltonian cycle. However, the converse of this statement is not necessarily true, because $G$ can admit 2-factors which do not contain the edge $e$. An example of such a graph is $K_{3,3}$.

\begin{theorem}\label{theorem 2ec2fh}
Let $G_{1}$ and $G_{2}$ be two Hamiltonian graphs such that $e_{1}\in E(G_{1})$ and $e_{2}\in E(G_{2})$. The graph $G=G_{1}(e_{1})\#G_{2}(e_{2})$ is 2FH if and only if every 2-factor of $G_{1}$ is a Hamiltonian cycle containing $e_{1}$ and every 2-factor of $G_{2}$ containing $e_{2}$ is a Hamiltonian cycle (or vice-versa). 
\end{theorem}

\begin{proof} 
($\Rightarrow$) Let the principal 2-edge-cut of $G$ be $X$. Since $G$ is 2FH, we cannot have that both $G_{1}$ and $G_{2}$ admit a 2-factor $F_{1}$ and $F_{2}$, respectively, such that $e_{1}\not\in E(F_{1})$ and $e_{2}\not\in E(F_{2})$, because otherwise, $F_{1}$ together with $F_{2}$ would form a 2-factor of $G$ which does not intersect its principal 2-edge-cut, and so is not a Hamiltonian cycle. Therefore, without loss of generality, we can assume that $e_1$ is in every 2-factor of $G_1$. Let $F'$ be a 2-factor of $G_1$, and let $F''$ be a 2-factor of $G_2$ which contains $e_2$. These 2-factors do exist since $G$ is 2FH. Clearly, $F' \cup F'' \cup X - \{e_1,e_2\}$ is a 2-factor of $G$, and since $G$ is 2FH, $F' \cup F'' \cup X - \{e_1,e_2\}$ is a Hamiltonian cycle of $G$. Consequently, $F'$ and $F''$ are Hamiltonian cycles of $G_1$ and $G_2$, respectively.

($\Leftarrow$) For each $i\in\{1,2\}$, let $e_i=x_iy_i$. Consider a 2-factor $F$ of $G$. Due to the condition on $G_{1}$, any 2-factor of $G$ must contain the principal 2-edge-cut of $G$, because otherwise, this would create a 2-factor in $G_{1}$ not containing the edge $e_{1}$, a contradiction. Thus, $X\subset E(F)$. By our assumptions on $G_1$ and $G_2$, it follows that $(F \cap G_i) \cup \{e_i\}$ is a (connected) 2-factor of $G_i$, for each $i\in\{1,2\}$. Consequently, for each $i\in\{1,2\}$, $F \cap G_i$ is a Hamiltonian path of $G_i$ with endvertices $x_i$ and $y_i$, implying that $F$ is a Hamiltonian cycle of $G$.
\end{proof}

We first note that the graph $G_{2}$ in the above theorem is not necessarily 2FH, and can admit a 2-factor which is not a Hamiltonian cycle. In fact, there exist 2FH-graphs which are obtained by applying a 2-cut connection between graphs which are not both 2FH. Let $G_{1}$ be the cycle on four vertices, and let $e_{1}$ be one of its edges. Let $G_{2}$ be the graph obtained by applying a star product between two copies of $K_{4}$, and let $e_{2}$ be one of the edges of the principal 3-edge-cut. The graph $G_{1}$ is 2FH, whilst $G_{2}$ is not, as already stated above. However, $G_{1}(e_{1})\#G_{2}(e_{2})$ is still 2FH (see Figure \ref{figure 2fh2ec}).

\begin{figure}[h]
\centering
\includegraphics[scale=1]{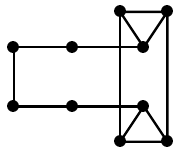}
\caption{A 2FH-graph arising from a 2-cut connection between two graphs one of which is not 2FH.}
\label{figure 2fh2ec}
\end{figure}

We further remark that, in the above theorem, the graphs $G_{1}$ and $G_{2}$ cannot both be just 2FH-graphs without any further properties, because a 2-cut connection between two copies of $K_{3,3}$ is not 2FH. Moreover, the condition on $G_{1}$ (that is, every 2-factor of $G_{1}$ is a Hamiltonian cycle containing $e_{1}$) cannot be relaxed to be equivalent to the condition on $G_{2}$ (that is, every 2-factor of $G_{2}$ containing $e_{2}$ is a Hamiltonian cycle). In fact, let $G_{1}$ and $G_{2}$ be two copies of the graph obtained by applying a star product between two copies of $K_{4}$, and let $e_{1}$ and $e_{2}$ be one of the edges of the principal 3-edge-cut of $G_{1}$ and $G_{2}$, respectively. For each $i\in\{1,2\}$, every 2-factor of $G_{i}$ containing $e_{i}$ is a Hamiltonian cycle of $G_{i}$, however, the graph $G_{1}(e_{1})\#G_{2}(e_{2})$, portrayed in Figure \ref{figure lastnot2fh}, is not 2FH.

\begin{figure}[h]
\centering
\includegraphics[scale=1]{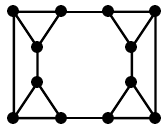}
\caption{A graph arising from a 2-cut connection which is not 2FH.}
\label{figure lastnot2fh}
\end{figure}

\section*{Acknowledgements}

The authors would like to thank Mari\'en Abreu, John Baptist Gauci, Domenico Labbate and Giuseppe Mazzuoccolo for useful discussions about the topic.

\end{document}